\documentclass[10pt]{amsart}
\usepackage{geometry}
\usepackage{amsmath, amsthm, amscd, amsfonts, amssymb, graphicx, xcolor}
\usepackage{hyperref}
\usepackage{mathtools}
\usepackage{enumitem}

\def\dsum{\displaystyle\sum}

\newtheorem{theorem}{Theorem}
\theoremstyle{plain}

\newtheorem{conjecture}{Conjecture}

\newtheorem{lemma}{Lemma}

\newtheorem{problem}{Problem}

\newtheorem{remark}{Remark}

\numberwithin{equation}{section}

\begin{document}
	\setcounter{page}{1}
	
\title[Properties of beta function and Ramanujan $R$-function]{Some new
	properties of the beta function and Ramanujan $R$-function}
\author{Zhen-Hang Yang}
\address{Zhen-Hang Yang, State Grid Zhejiang Electric Power Company Research
	Institute, Hangzhou, Zhejiang, China, 310014}
\email{yzhkm@163.com}
\urladdr{https://orcid.org/0000-0002-2719-4728}
\author{Miao-Kun Wang$^{**}$}
\address{Miao-Kun Wang, Department of mathematics, Huzhou University,
	Huzhou, Zhejiang, China, 313000}
\email{wmk000@126.com}
\urladdr{https://orcid.org/0000-0002-0895-7128}
\author{Tie-Hong Zhao}
\address{Tie-Hong Zhao, School of mathematics, Hangzhou Normal University,
	Hangzhou, Zhejiang, China, 311121}
\email{tiehong.zhao@hznu.edu.cn}
\urladdr{https://orcid.org/0000-0002-6394-1094}
\date{June 23, 2024}
\thanks{This research was supported by the National Natural Science
	Foundation of China (11971142) and the Natural Science Foundation of
	Zhejiang Province (LY24A010011).}
\thanks{$^{**}$Corresponding author}

	\begin{abstract}
		In this paper, the power series and hypergeometric series representations of
		the beta and Ramanujan functions 
		\begin{equation*}
			\mathcal{B}\left( x\right) =\frac{\Gamma \left( x\right) ^{2}}{\Gamma \left(
				2x\right) }\text{ \ and \ }\mathcal{R}\left( x\right) =-2\psi \left(
			x\right) -2\gamma
		\end{equation*}%
		are presented, which yield higher order monotonicity results related to $%
		\mathcal{B}(x)$ and $\mathcal{R}(x)$; the decreasing property of the
		functions $\mathcal{R}\left( x\right) /\mathcal{B}\left( x\right) $ and $[%
		\mathcal{B}(x) -\mathcal{R}(x)] /x^{2}$ on $\left( 0,\infty \right) $ are
		proved. Moreover, a conjecture put forward by Qiu et al. in \cite%
		{Qiu-JMAA-446-2017} is proved to be true. As applications, several
		inequalities and identities are deduced. These results obtained in this
		paper may be helpful for the study of certain special functions. Finally, an
		interesting infinite series similar to Riemann zeta functions is observed
		initially.
	\end{abstract}
	
	\maketitle
	
	\section{Introduction}
	
	For $a,b,c\in \mathbb{R}$ with $-c\notin \mathbb{N}\backslash \left\{
	0\right\} $, the hypergeometric function is defined on $\left( -1,1\right) $
	by%
	\begin{equation}
		F\left( a,b;c;x\right) =\sum_{n=0}^{\infty }\frac{\left( a\right) _{n}\left(
			b\right) _{n}}{\left( c\right) _{n}}\frac{x^{n}}{n!},  \label{2F1-1}
	\end{equation}%
	where $\left( a\right) _{n}=\Gamma \left( n+a\right) /\Gamma \left( a\right) 
	$ and $\Gamma \left( x\right) $ is the classical gamma function. In
	particular, $F\left( a,b;a+b;x\right) $ is called zero-balanced
	hypergeometric function, which satisfies the asymptotic formula%
	\begin{equation}
		F\left( a,b;a+b;x\right) =\frac{R\left( a,b\right) -\ln t}{B\left(
			a,b\right) }+O\left( t\ln t\right) \text{ \ as }t=1-x\rightarrow 0\text{,}
		\label{F-af-R}
	\end{equation}%
	where%
	\begin{equation*}
		B\left( a,b\right) =\frac{\Gamma \left( a\right) \Gamma \left( b\right) }{%
			\Gamma \left( a+b\right) }\text{ \ and \ }R\left( a,b\right) =-\psi \left(
		a\right) -\psi \left( b\right)-2\gamma
	\end{equation*}%
	are the beta function and Ramanujan $R$-function (or Ramanujan constant),
	respectively, while $\psi \left( x\right) =d[\ln \Gamma \left( x\right) ]/dx$
	is the psi function and $\gamma =-\psi (1)$ is the Euler constant. The
	asymptotic formula \eqref{F-af-R} was due to Ramanujan \cite%
	{Berndt-RN-P2-SV-1989}.
	
	In the study for special functions in the geometric function theory and
	Ramanujan's modular equation, it is often crucial to the treatments for the
	functions $\left( a,b\right) \mapsto B(a,b)$, $\left( a,b\right) \mapsto
	R\left( a,b\right) $ and their combinations. For this, ones have to spend a
	lot of time dealing with such problems. For example, in order to prove the
	double inequality 
	\begin{equation}
		1+\alpha (1-r^{2})<\frac{K_{a}\left( r\right) }{\sin \left( \pi a\right) \ln
			\left( e^{R\left( a\right) /2}/\sqrt{1-r^{2}}\right) }<1+\beta (1-r^{2})
		\label{I-WCQ}
	\end{equation}%
	holds for all $a\in (0,1/2]$ and $r\in \left( 0,1\right) $ if and only if $%
	\alpha \leq \alpha _{0}=\pi /\left[ R\left( a\right) \sin \left( \pi
	a\right) \right] -1$\emph{\ and }$\beta \geq \beta _{0}=a\left( 1-a\right) $
	where 
	\begin{equation*}
		K_{a}\left( r\right) =\frac{\pi }{2}F\left( a,1-a;1,r^{2}\right) \text{\ and
			\ }R\left( a\right) =R(a,1-a),
	\end{equation*}%
	five properties of $R\left( a\right) $ should be proved at first (see \cite[%
	Section 2]{Wang-JMAA-429-2015}). For another example, it is precisely
	because the inequality established in \cite[Eq. (2.13)]{Wang-JMAA-429-2015},
	the increasing property of the function%
	\begin{equation}
		r\mapsto Y\left( r\right) =\frac{K_{a}\left( r\right) }{(1-r^{2})\sin \left(
			\pi a\right) \ln \left( e^{R\left( a\right) /2}/\sqrt{1-r^{2}}\right) }-%
		\frac{1}{1-r^{2}}  \label{Y(r)}
	\end{equation}%
	on $\left( 0,1\right) $ was proved (see \cite{Yang-MIA-20-2017-3}). More of
	such examples can be seen in \cite{Evans-SIAM0JMA-15-1984}, \cite%
	{Ponnusamy-M-44-1997}, \cite{Qiu-NMJ-154-1999}, \cite[Lemma 2.14]%
	{Qiu-FM-12-2000}, \cite{Heikkala-JMAA-338-2008}, \cite{Heikkala-CMFT-9-2009}%
	, \cite[Lemma 3.5]{Wang-MIA-22-2019}, \cite{Yang-RJ-48-2019}, \cite%
	{Wang-AMS-39B-2019}. Fortunately, several properties of $R\left(
	a,1-a\right) $, $B\left( a,1-a\right) $ and their combinations have been
	discovered, which can be summarized as the following three aspects:
	\begin{itemize}[leftmargin=2.5em]
		\item[(i)] the monotonicity and bounds for $R\left( x\right) =R\left( x,1-x\right) $
		can be found in \cite[Lemma 2.1]{Qiu-AMS-43-2000}, \cite{Qiu-JGDU-27-2007}, 
		\cite[Theorem 1]{Zhou-JZSTU-27-2010}, \cite[Lemma 2.1, Theorem 2.3,
		Corollaries 2.4 and 2.5]{Wang-JMAA-429-2015}, \cite{Chu-JIA-2016-196}, \cite[%
		Lemma 4]{Yang-RJ-48-2019};
		\item [(ii)] the monotonicity and convexity results for the combinations of $%
		R\left(x\right) $ and $B\left(x\right)=B\left( x,1-x\right) $ can be seen in 
		\cite[Theorem 2]{Zhou-JZSTU-27-2010}, \cite[Theorems 1.1--1.3]%
		{Qiu-CA-51-2020};
		\item [(iii)] the series expansion of $R\left(x\right)$ was presented in \cite[%
		Theorem 2.2]{Wang-JMAA-429-2015}, the power series expansions of $R\left(
		x\right) -B\left( x\right) $, $\left[ R\left( x\right) -B\left( x\right) %
		\right] /x$ and $\left[ R\left( x\right) -B\left( x\right) \right]/\left[
		x\left( 1-x\right) \right] $ at $x=0$, $1/2$ and higher monotonicity of them
		were established in \cite[Theorems 1.1 and 1.2]{Qiu-JMAA-446-2017}, while
		the power series expansions of $[1+x(1-x)]R(x)-B(x)$ at $x=0$, $1/2$ and its
		complete monotonicity were found in \cite[Theorem 1.1]{Qiu-CA-51-2020}.
	\end{itemize}
	
In the end of \cite{Qiu-CA-51-2020}, Qiu, Ma and Huang  made a conjecture below.
	
	\begin{conjecture}[{\protect\cite[Conjecture 5.3]{Qiu-CA-51-2020}}]
		\label{C-Qiu} The function%
		\begin{equation*}
			F\left( x\right) =R\left( x\right) -\frac{B\left( x\right) }{1+x\left(
				1-x\right) }
		\end{equation*}%
		is completely monotonic on $\left( 0,1/2\right) $.
	\end{conjecture}
	
	Now let us consider another case in $R\left( a,b\right) $ and $B\left(
	a,b\right) $ with $a=b=x\in (0,\infty )$. In this case, we denote by%
	\begin{align}
		\mathcal{B}\left( x\right) & =B\left( x,x\right) =\frac{\Gamma \left(
			x\right) ^{2}}{\Gamma \left( 2x\right) },  \label{B} \\
		\mathcal{R}\left( x\right) & =R\left( x,x\right) =-2\psi \left( x\right)
		-2\gamma .  \label{R}
	\end{align}%
	The aim of this paper are twofold. One is to establish some new analytic
	properties of the above two functions. The other is to prove the
	Conjecture \ref{C-Qiu} to be true. Specifically, we shall present several
	series expansions and higher order monotonicity properties of $\mathcal{R}(x)
	$, $\mathcal{B}(x)$ and some functions involving them in Section \ref{sec-3}.
	Furthermore, the monotonicity of $x\mapsto \mathcal{R}\left( x\right) /%
	\mathcal{B}\left( x\right) $, $(\mathcal{B}(x)-\mathcal{R}(x))/x^{2}$ on $%
	(0,\infty )$ will be proved subsequently in Section \ref{sec-4}. In Section \ref{sec-5}, we will
	give an affirmative answer to the Conjecture \ref{C-Qiu}. In the last
	section, several remarks are listed.
	
	\bigskip
	
	\section{Preliminaries}
	
	In order to establish some new properties of $\mathcal{B}\left(x\right) $
	and $\mathcal{R}\left( x\right)$, as well as, combinations of them, we need
	some basic knowledge points below.
	\begin{enumerate}[leftmargin=2.2em]
		\item Duplication formulas for the gamma, psi and polygamma functions (\cite[%
	p. 256, Eq. (6.1.18); p. 259. Eq. (6.3.8)]{Abramowitz-HMFFGMT-1970}):%
	\begin{align}
		\Gamma \left( 2x\right) &=\frac{2^{2x-1}}{\sqrt{\pi }}\Gamma \left( x+\frac{%
			1}{2}\right) \Gamma \left( x\right) ,  \label{g-df} \\
		\psi \left( 2x\right) &=\frac{1}{2}\psi \left( x+\frac{1}{2}\right) +\frac{1%
		}{2}\psi \left( x\right) +\ln 2,  \label{psi-df} \\
		2^{n}\psi ^{\left( n\right) }\left( 2x\right) &=\frac{1}{2}\psi ^{\left(
			n\right) }\left( x+\frac{1}{2}\right) +\frac{1}{2}\psi ^{\left( n\right)
		}\left( x\right) \text{ \ for }n\in \mathbb{N}.  \label{pn-df}
	\end{align}
	\item The series and integral representations of psi and polygamma functions (%
	\cite[p. 259-260, Eq. (6.3.22), Eq. (6.3.16), Eq. (6.4.10)]{Abramowitz-HMFFGMT-1970}): 
	\begin{equation}
		\psi (x)+\gamma =\int\limits_{0}^{\infty }\frac{e^{-t}-e^{-xt}}{1-e^{-t}}dt
		\label{psi-ir}
	\end{equation}and
	\begin{equation}
		\psi ^{(n)}(x)=\begin{dcases}
			-\gamma -\dfrac{1}{x}+\sum\limits_{k=1}^{\infty }\dfrac{x}{k(k+x)} & \text{%
			for}\quad n=0, \\ 
		(-1)^{n+1}n!\sum\limits_{k=0}^{\infty }\dfrac{1}{(x+k)^{n+1}} & \text{for}%
		\quad n\in \mathbb{N}. 
		\end{dcases}
		 \label{psi-sr}
	\end{equation}
	\item  Recurrence formula for the psi and polygamma functions (\cite[p. 260,
	Eq. (6.4.6)]{Abramowitz-HMFFGMT-1970}):%
	\begin{equation}
		\psi ^{(n)}(x+1)-\psi ^{(n)}(x)=\left( -1\right) ^{n}\frac{n!}{x^{n+1}}\text{
			\ for }n\in \mathbb{N}_{0}:\mathbb{=N\cup }\left\{ 0\right\} .  \label{pn-rf}
	\end{equation}
	\item The Wallis ratio $W_{n}$ is defined on $n\in \mathbb{N}_{0}$ by 
	\begin{equation}
		W_{n}=\frac{\left( 2n-1\right) !!}{\left( 2n\right) !!}=\frac{\left(
			1/2\right) _{n}}{n!}=\frac{\Gamma \left( n+1/2\right) }{\Gamma \left(
			1/2\right) \Gamma \left( n+1\right) },  \label{Wn}
	\end{equation}
	which satisfies the following inequality (see \cite{Chen-PAMS-133-2005})
	\begin{equation}
		\frac{1}{\sqrt{\pi \left( n+4/\pi -1\right) }}<W_{n}<\frac{1}{\sqrt{\pi
				\left( n+1/4\right) }}  \label{Wn<>}\quad \text{for}\ n\in\mathbb{N}.
	\end{equation}%
	\item The asymptotic formulas for the gamma, psi and polygamma functions: as $x\rightarrow \infty $,%
	\begin{align}
		\Gamma \left( x\right)  &\thicksim \sqrt{2\pi }x^{x-1/2}e^{-x},\qquad \psi \left( x\right) \thicksim \ln x-\frac{1}{2x},  \label{g,psi-af} \\
		\psi ^{\left( n\right) }\left( x\right)  &\thicksim \frac{\left( -1\right)
			^{n-1}\left( n-1\right) !}{x^{n}}+\frac{1}{2}\frac{\left( -1\right) ^{n-1}n!%
		}{x^{n+1}} \label{pg-af}
	\end{align}%
	(\cite[ps. 258--260, Eqs. (6.1.40), (6.3.18), (6.4.11)]%
	{Abramowitz-HMFFGMT-1970}).
	\item The value of hypergeometric function at $x=1$: for $c-a-b>0$,%
	\begin{equation}
		F\left( a,b;c;1\right) =\frac{\Gamma \left( c\right) \Gamma \left(
			c-a-b\right) }{\Gamma \left( c-a\right) \Gamma \left( c-b\right) }
		\label{F-1}
	\end{equation}%
	(\cite[p. 556, Eq. (15.1.20)]{Abramowitz-HMFFGMT-1970}).
	\item The Bernoulli polynomial $B_{n}\left( x\right) $ is defined by%
	\begin{equation*}
		\frac{te^{xt}}{e^{t}-1}=\sum_{n=0}^{\infty }B_{n}\left( x\right) \frac{t^{n}%
		}{n!}\text{ \ \ (}\left\vert t\right\vert <2\pi \text{) \ (\cite[p. 802, Eq.
			(23.1.1)]{Abramowitz-HMFFGMT-1970}),}
	\end{equation*}%
	while the Euler polynomial $E_{n}\left( x\right) $ is defined as%
	\begin{equation*}
		\frac{2e^{xt}}{e^{t}+1}=\sum_{n=0}^{\infty }E_{n}\left( x\right) \frac{t^{n}%
		}{n!}\text{ \ \ (}\left\vert t\right\vert <\pi \text{) \ (\cite[p. 802, Eq.
			(23.1.1)]{Abramowitz-HMFFGMT-1970})}.
	\end{equation*}%
	The Euler number $E_{n}$ is given by%
	\begin{equation*}
		E_{n}=2^{n}E_{n}\left( \frac{1}{2}\right) \text{ \ for }n\in \mathbb{N}_{0}%
		\text{,}
	\end{equation*}%
	which satisfies the double inequality 
	\begin{equation}
		\frac{4^{n+1}\left( 2n\right) !}{\pi ^{2n+1}}\frac{1}{1+3^{-2n-1}}<\left(
		-1\right) ^{n}E_{2n}<\frac{4^{n+1}\left( 2n\right) !}{\pi ^{2n+1}}\text{ \
			for }n\in \mathbb{N}_{0}  \label{E2n<}
	\end{equation}%
	\ (\cite[p. 805, Eq. (23.1.15)]{Abramowitz-HMFFGMT-1970}).
	\item Riemann zeta function is defined by%
	\begin{equation*}
		\zeta \left( s\right) =\sum_{n=1}^{\infty }\frac{1}{n^{s}}\text{ \ (}{\Re }%
		\left( s\right) >1\text{) \ \ \ (\cite[p. 807, Eq. (23.2.1)]%
			{Abramowitz-HMFFGMT-1970}).}
	\end{equation*}%
	Moreover, we have%
	\begin{align}
		\lambda \left( n\right) & =\sum_{k=0}^{\infty }\frac{1}{\left( 2k+1\right)
			^{n}}=\left( 1-2^{-n}\right) \zeta \left( n\right)\qquad\quad  \text{ for }n\in \mathbb{N%
		}\text{ with }n\geq 2,  \label{ln} \\
		\beta \left( 2n+1\right) & =\sum_{k=0}^{\infty }\frac{\left( -1\right) ^{k}}{%
			\left( 2k+1\right) ^{2n+1}}=\frac{\left( \pi /2\right) ^{2n+1}}{2\left(
			2n\right) !}\left\vert E_{2n}\right\vert\quad  \text{ for }n\in \mathbb{N}_{0}
		\label{b2n+1}
	\end{align}%
	(\cite[p. 807, (23.2.20), (23.2.22)]{Abramowitz-HMFFGMT-1970}).
	\item The integral representation of remainder term of Taylor formula:%
	\begin{align}
		f\left( x\right) -\sum_{k=0}^{n}\frac{f^{\left( k\right) }\left(
			x_{0}\right) }{k!}\left( x-x_{0}\right) ^{k}& =\frac{1}{n!}%
		\int_{x_{0}}^{x}\left( x-t\right) ^{n}f^{\left( n+1\right) }\left( t\right)
		dt  \notag \\
		& =\frac{\left( x-x_{0}\right) ^{n+1}}{n!}\int_{0}^{1}\left( 1-u\right)
		^{n}f^{\left( n+1\right) }\left( \xi \right) du,  \label{rtf-ir}
	\end{align}%
	where $\xi =ux+\left( 1-u\right) x_{0}$.
	\item  L'Hospital monotonic rule (\cite[Lemma 1.1]{Vamanamurthy-JMAA-183-1994}%
	), that is, the following lemma.
\end{enumerate}
	\begin{lemma}
		\label{L-LMR}Let $-\infty <a<b<\infty $, and let $f,g:[a,b]\rightarrow 
		\mathbb{R}$ be continuous functions that are differentiable on $\left(
		a,b\right) $, with $f\left( a\right) =g\left( a\right) =0$ or $f\left(
		b\right) =g\left( b\right) =0$. Assume that $g^{\prime }(x)\neq 0$ for each $%
		x$ in $(a,b)$. If $f^{\prime }/g^{\prime }$ is increasing (decreasing) on $%
		(a,b)$ then so is $f/g$.
	\end{lemma}
	
	\section{Expansions and complete monotonicity of $\mathcal{B}\left( x\right) 
		$ and $\mathcal{R}\left( x\right) $}\label{sec-3}
	
	We first give the higher order monotonicity of the function $x\mapsto \ln %
	\left[ x\mathcal{B}\left( x\right) \right] $ on $\left( 0,\infty \right) $.
	
	\begin{theorem}
		\label{T-RlnxB-sc}Let $\mathcal{B}\left( x\right) $ be defined by (\ref{B}).
		The following statements are valid.
		
		(i) $\mathcal{B}\left( x\right) $ satisfies%
		\begin{equation*}
			\left( -1\right) ^{n}\left[ \ln \left( \frac{x}{2}\mathcal{B}\left( x\right)
			\right) \right] ^{\left( n\right) }\left\{ 
			\begin{array}{ll}
				<0 & \text{for }n=0, \\ 
				>0 & \text{for }n=1, \\ 
				<0 & \text{for }n\geq 2.%
			\end{array}%
			\right.
		\end{equation*}%
		Consequently, the function $x\mapsto -\left[ \ln \left( x\mathcal{B}\left(
		x\right) \right) \right] ^{\prime \prime }$ is completely monotonic on $%
		\left( 0,\infty \right) $.
		
		(ii) $\ln \mathcal{B}\left( x\right) $ has the power series expansion 
		\begin{equation}
			\ln \mathcal{B}\left( x\right) =\ln 2-\ln x+\sum_{n=2}^{\infty }\left(
			-1\right) ^{n-1}\frac{\left( 2^{n}-2\right) \zeta \left( n\right) }{n}x^{n},
			\label{lnB-ps}
		\end{equation}%
		which is convergent on $\left( 0,1/2\right]$.
		
		(iii) For $n\geq 2$, the function 
		\begin{equation*}
			x\mapsto \left( -1\right) ^{n}\frac{\ln \mathcal{B}\left( x\right) +\ln
				x-\ln 2-\sum_{k=2}^{n}\left( -1\right) ^{k-1}\left( 2^{k}-2\right)
				k^{-1}\zeta \left( k\right) x^{k}}{x^{n+1}}
		\end{equation*}%
		is completely monotonic on $\left( 0,\infty \right) $.
	\end{theorem}
	
	\begin{proof}
		(i) Employing the duplication formula (\ref{g-df}), we have 
		\begin{equation}
			\mathcal{B}_{\ell }\left( x\right) :\mathcal{=}\ln \left( \frac{x}{2}%
			\mathcal{B}\left( x\right) \right) =\frac{1}{2}\ln \pi -2x\ln 2+\ln \Gamma
			\left( x+1\right) -\ln \Gamma \left( x+\frac{1}{2}\right) .  \label{Bl}
		\end{equation}%
		Differentiating and then using the integral representation of (\ref{psi-ir})
		yield 
		\begin{align*}
			{\mathcal{B}_{\ell }}^{\prime }(x)& =-2\ln 2+\psi \left( x+1\right) -\psi
			\left( x+\frac{1}{2}\right) =-2\ln 2+\int_{0}^{\infty }\frac{e^{-xt}}{%
				e^{t/2}+1}dt, \\
			{\mathcal{B}_{\ell }}^{\prime \prime }(x)& =-\int_{0}^{\infty }\frac{te^{-xt}%
			}{e^{t/2}+1}dt<0.
		\end{align*}%
		So that 
		\begin{align*}
			{\mathcal{B}_{\ell }}^{\prime }(x)& <\lim_{x\rightarrow 0^{+}}\left[ -2\ln
			2+\psi \left( x+1\right) -\psi \left( x+\frac{1}{2}\right) \right] =0, \\
			{\mathcal{B}_{\ell }}(x)& <\lim_{x\rightarrow 0^{+}}\left[ \frac{1}{2}\ln
			\pi -2x\ln 2+\ln \Gamma \left( x+1\right) -\ln \Gamma \left( x+\frac{1}{2}%
			\right) \right] =0
		\end{align*}%
		for all $x\in (0,\infty )$. Further, we have%
		\begin{equation*}
			\left( -1\right) ^{n}\mathcal{B}_{\ell }^{(n)}\left( x\right)
			=-\int_{0}^{\infty }\frac{t^{n-1}e^{-xt}}{e^{t/2}+1}dt<0\text{ \ for }n\geq
			2.
		\end{equation*}
		
		(ii) Note that 
		\begin{equation*}
			\mathcal{B}_{\ell }\left( x\right) =\ln \left( \frac{x}{2}\mathcal{B}\left(
			x\right) \right) =\ln \frac{x\Gamma \left( x\right) ^{2}}{2\Gamma \left(
				2x\right) }=2\ln \Gamma \left( x+1\right) -\ln \Gamma \left( 2x+1\right) .
		\end{equation*}%
		Differentiation leads to 
		\begin{equation*}
			\mathcal{B}_{\ell }^{\left( n\right) }\left( x\right) =2\psi ^{\left(
				n-1\right) }\left( x+1\right) -2^{n}\psi ^{\left( n-1\right) }\left(
			2x+1\right) \text{ \ for }n\geq 1.
		\end{equation*}%
		Hence, $\mathcal{B}_{\ell }\left( 0\right) =\mathcal{B}_{\ell }^{\prime
		}\left( 0\right) =0$ and by \eqref{psi-sr}, 
		\begin{align*}
			\mathcal{B}_{\ell }^{\left( n\right) }\left( 0\right) & =\left(
			2-2^{n}\right) \psi ^{\left( n-1\right) }\left( 1\right) =\left(
			2-2^{n}\right) \left( -1\right) ^{n}\left( n-1\right)
			!\sum\limits_{k=0}^{\infty }\frac{1}{(k+1)^{n}} \\
			& =\left( 2-2^{n}\right) \left( -1\right) ^{n}\left( n-1\right) !\zeta
			\left( n\right) \text{ \ for }n\geq 2.
		\end{align*}%
		It then follows that%
		\begin{equation*}
			\mathcal{B}_{\ell }\left( x\right)=\mathcal{B}_{\ell }\left( 0\right) +%
			\mathcal{B}_{\ell }^{\prime }\left( 0\right) x+\sum_{n=2}^{\infty }\frac{%
				\mathcal{B}_{\ell }^{\left( n\right) }\left( 0\right) }{n!}x^{n}
			=\sum_{n=2}^{\infty }\left( -1\right) ^{n-1}\frac{\left( 2^{n}-2\right)
				\zeta \left( n\right) }{n}x^{n},
		\end{equation*}%
		which is convergent on $\left( 0,1/2\right] $.
		
		(iii) By the integral representation of the reminder term of the Taylor
		formula (\ref{rtf-ir}), we have 
		\begin{equation*}
			\left( -1\right) ^{n}\frac{\mathcal{B}_{\ell }\left( x\right) -\sum_{k=2}^{n}%
				\frac{\mathcal{B}_{\ell }^{\left( k\right) }\left( 0\right) }{k!}x^{k}}{%
				x^{n+1}}=-\int_{0}^{1}\frac{\left( 1-u\right) ^{n}}{n!}\left( -1\right)
			^{n+1}\mathcal{B}_{\ell }^{\left( n+1\right) }\left( ux\right) du.
		\end{equation*}%
		By the assertion (i) of this theorem, the desired complete monotonicity
		follows. This completes the proof.
	\end{proof}
	
	The following theorem reveals the hypergeometric series and power series
	representations of $\mathcal{B}\left( x\right) $.
	
	\begin{theorem}
		\label{T-B-hs}The following statements are valid.
		\begin{enumerate}[leftmargin=2.2em,label=(\roman*)]
			\item The expansions 
			\begin{equation}
				\mathcal{B}\left( x\right) =2^{1-2x}\sum_{n=0}^{\infty }\frac{W_{n}}{x+n}%
				=2\sum_{n=0}^{\infty }\frac{\left( 1-x\right) _{n}}{\left( x+n+1\right) n!}.
				\label{B-hs}
			\end{equation}%
			hold for $x>0$, where $W_{n}$ is the Wallis ratio defined by (\ref{Wn}).
			\item The functions $\mathcal{B}$ and 
			\begin{equation}
				x\mapsto \mathcal{B}r_{1}\left( x\right) =2^{2x-1}\mathcal{B}\left( x\right)
				-\sum_{k=0}^{n}\frac{W_{k}}{x+k}  \label{Br1}
			\end{equation}%
			is completely monotonic on $\left( 0,\infty \right) $, while the function 
			\begin{equation*}
				x\mapsto \mathcal{B}r_{2}\left( x\right) =\mathcal{B}\left( x\right)
				-2\sum_{k=0}^{n}\frac{\left( 1-x\right) _{k}}{\left( x+k+1\right) k!}
			\end{equation*}%
			is completely monotonic on $\left( 0,1\right) $.
			\item For $x\in \left( 0,1\right) $, the function $\mathcal{B}\left( x\right) 
			$ has the power series representation 
			\begin{equation}
				\mathcal{B}\left( x\right) =\frac{2}{x}\sum_{k=0}^{\infty }a_{k}x^{k},
				\label{B-ps}
			\end{equation}%
			where $a_{0}=1$ and for $k\geq 1$, 
			\begin{equation}
				a_{k}=\left( -1\right) ^{k}\left[ \frac{\left( \ln 4\right) ^{k}}{k!}%
				-\sum_{j=0}^{k-1}\left( \sum_{n=1}^{\infty }\frac{W_{n}}{n^{k-j}}\right) 
				\frac{\left( \ln 4\right) ^{j}}{j!}\right] .  \label{ak}
			\end{equation}
		\end{enumerate}
	\end{theorem}
	
	\begin{proof}
		(i) Using the duplication formula (\ref{g-df}) and \eqref{2F1-1} gives 
		\begin{align*}
			2^{2x-1}\frac{x\Gamma \left( x\right) ^{2}}{\Gamma \left( 2x\right) }& =%
			\sqrt{\pi }\frac{\Gamma \left( x+1\right) }{\Gamma \left( x+\frac{1}{2}%
				\right) }=\frac{\Gamma \left( 1/2\right) \Gamma \left( x+1\right) }{\Gamma
				\left( 1\right) \Gamma \left( x+1/2\right) } \\
			& =F\left( \frac{1}{2},x;x+1;1\right) =\sum_{n=0}^{\infty }\frac{\left(
				1/2\right) _{n}\left( x\right) _{n}}{\left( x+1\right) _{n}}\frac{1^{n}}{n!}%
			=\sum_{n=0}^{\infty }\frac{xW_{n}}{x+n},
		\end{align*}%
		which implies the first series of (\ref{B-hs}). Similarly, we have 
		\begin{align*}
			\frac{\Gamma \left( x\right) ^{2}}{\Gamma \left( 2x\right) } =&\frac{2}{x+1}%
			\frac{\Gamma \left( x\right) \Gamma \left( x+2\right) }{\Gamma \left(
				1\right) \Gamma \left( 2x+1\right) }=\frac{2}{x+1}F\left(
			x+1,1-x;x+2;1\right) = \\
			=&\frac{2}{x+1}\sum_{n=0}^{\infty }\frac{\left( x+1\right) _{n}\left(
				1-x\right) _{n}}{\left( x+2\right) _{n}}\frac{1^{n}}{n!}=2\sum_{n=0}^{\infty
			}\frac{\left( 1-x\right) _{n}}{\left( x+n+1\right) n!},
		\end{align*}%
		which proves the second series of (\ref{B-hs}).
		
		(ii) Since $2^{1-2x}=2\exp \left( -x\ln 4\right) $ and $1/\left( x+n\right) $
		for every $n\geq 0$ are completely monotonic on $\left( 0,\infty \right) $,
		so is the function $\mathcal{B}$ by the first series of (\ref{B-hs}). Note
		that 
		\begin{equation*}
			\mathcal{B}r_{1}\left( x\right) =2^{2x-1}\mathcal{B}\left( x\right)
			-\sum_{k=0}^{n}\frac{W_{k}}{x+k}=\sum_{k=n+1}^{\infty }\frac{W_{k}}{x+k}.
		\end{equation*}%
		A similar argument leads to the complete monotonicity of $\mathcal{B}%
		r_{1}\left( x\right) $ on $\left( 0,\infty \right) $. Obviously, $1/\left(
		x+k+1\right) $ for every $k\geq 0$ and $1-x+j$ for every $0\leq j\leq k-1$
		are completely monotonic on $\left( 0,1\right) $, so is $\mathcal{B}%
		r_{2}\left( x\right) $ because of 
		\begin{equation*}
			\mathcal{B}r_{2}\left( x\right) =\mathcal{B}\left( x\right) -2\sum_{k=0}^{n}%
			\frac{\left( 1-x\right) _{k}}{\left( x+k+1\right) k!}=2\sum_{k=n+1}^{\infty }%
			\frac{\left( 1-x\right) _{k}}{\left( x+k+1\right) k!}.
		\end{equation*}
		
		(iii) Write 
		\begin{align*}
			\sum_{n=0}^{\infty }\frac{xW_{n}}{x+n}& =1+\sum_{n=1}^{\infty }\frac{W_{n}}{n%
			}\frac{x}{1+x/n}=1+\sum_{n=1}^{\infty }\frac{W_{n}}{n}\left(
			\sum_{k=0}^{\infty }\frac{\left( -1\right) ^{k}}{n^{k}}x^{k+1}\right) \\
			& =1+\sum_{k=1}^{\infty }\left( -1\right) ^{k-1}\left( \sum_{n=1}^{\infty }%
			\frac{W_{n}}{n^{k}}\right) x^{k}, \\
			2^{-2x}& =\exp \left( -x\ln 4\right) =\sum_{k=0}^{\infty }\frac{\left(
				-1\right) ^{k}\left( \ln 4\right) ^{k}}{k!}x^{k}.
		\end{align*}%
		Then by Cauchy product formula, we obtain 
		\begin{align*}
			\frac{x}{2}\mathcal{B}\left( x\right) & =2^{-2x}\sum_{n=0}^{\infty }\frac{%
				xW_{n}}{x+n}=\left( \sum_{k=0}^{\infty }\frac{\left( -1\right) ^{k}\left(
				\ln 4\right) ^{k}}{k!}x^{k}\right) \left( 1+\sum_{k=1}^{\infty }\left(
			-1\right) ^{k-1}\left( \sum_{n=1}^{\infty }\frac{W_{n}}{n^{k}}\right)
			x^{k}\right) \\
			& =1+\sum_{k=1}^{\infty }\left( -1\right) ^{k}\left[ \frac{\left( \ln
				4\right) ^{k}}{k!}-\sum_{j=0}^{k-1}\left( \sum_{n=1}^{\infty }\frac{W_{n}}{%
				n^{k-j}}\right) \frac{\left( \ln 4\right) ^{j}}{j!}\right] x^{k},
		\end{align*}%
		which completes the proof.
	\end{proof}
	
	\begin{remark}
		The coefficients $a_{k}$ in (\ref{B-ps}) can be given by recurrence relation
		using (\ref{lnB-ps}). In fact, by (\ref{lnB-ps}), we have%
		\begin{equation*}
			\frac{x}{2}\mathcal{B}\left( x\right) =\exp \left( \sum_{n=2}^{\infty
			}\left( -1\right) ^{n-1}\frac{\left( 2^{n}-2\right) \zeta \left( n\right) }{n%
			}x^{n}\right) =\sum_{n=0}^{\infty }a_{n}x^{n}.
		\end{equation*}%
		Then $a_{0}=1$. Differentiation yields 
		\begin{equation*}
			\left( \sum_{n=2}^{\infty }\left( -1\right) ^{n-1}\left( 2^{n}-2\right)
			\zeta \left( n\right) x^{n-1}\right) \exp \left( \sum_{n=2}^{\infty }\left(
			-1\right) ^{n-1}\frac{\left( 2^{n}-2\right) \zeta \left( n\right) }{n}%
			x^{n}\right) =\sum_{n=0}^{\infty }na_{n}x^{n-1},
		\end{equation*}%
		that is, 
		\begin{equation*}
			\left( \sum_{n=2}^{\infty }\left( -1\right) ^{n-1}\left( 2^{n}-2\right)
			\zeta \left( n\right) x^{n}\right) \left( \sum_{n=0}^{\infty
			}a_{n}x^{n}\right) =\sum_{n=0}^{\infty }na_{n}x^{n}.
		\end{equation*}%
		Using Cauchy product formula and comparing the coefficients of $x^{n}$ gives 
		$a_{1}=0$, and for $n\geq 2$,%
		\begin{equation}
			a_{n}=\frac{1}{n}\sum_{k=2}^{n}\left( -1\right) ^{k-1}\left( 2^{k}-2\right)
			\zeta \left( k\right) a_{n-k}.  \label{an-rr}
		\end{equation}%
		In particular, $a_{2}=-\pi ^{2}/6$, $a_{3}=2\zeta \left( 3\right) $, $%
		a_{4}=-\pi ^{4}/40$.
	\end{remark}
	
	\begin{remark}
		\label{Remark2} From the proof of Theorem \ref{T-B-hs}, we easily find that 
		\begin{align*}
			2^{2x-1}\mathcal{B}\left( x\right) -\frac{1}{x}=&\sum_{n=1}^{\infty }\frac{%
				W_{n}}{x+n}=\sum_{n=1}^{\infty }\frac{W_{n}}{n}\left( \sum_{k=0}^{\infty }%
			\frac{\left( -1\right) ^{k}}{n^{k}}x^{k}\right) \\
			=&\sum_{k=0}^{\infty }\left( -1\right) ^{k}\left( \sum_{n=1}^{\infty }\frac{%
				W_{n}}{n^{k+1}}\right) x^{k}.
		\end{align*}
	\end{remark}
	
	The power series representation of $\mathcal{R}\left( x\right) $ on $\left(
	0,1\right) $ is presented in the following theorem.
	
	\begin{theorem}
		\label{T-R-sc}Let $x\in \left( 0,1\right) $. We have 
		\begin{equation}
			\mathcal{R}\left( x\right) =-2\psi \left( x\right) -2\gamma =\frac{2}{x}%
			+2\sum_{n=1}^{\infty }\left( -1\right) ^{n}\zeta \left( n+1\right) x^{n}.
			\label{R-ps}
		\end{equation}%
		Moreover, for $n\geq 1$, the function 
		\begin{equation}
			x\mapsto \mathcal{R}r_{n}\left( x\right) =\frac{\left( -1\right) ^{n-1}}{%
				x^{n+1}}\left[ \mathcal{R}\left( x\right) -\frac{2}{x}-2\sum_{k=1}^{n}\left(
			-1\right) ^{k}\zeta \left( k+1\right) x^{k}\right]  \label{Rrn}
		\end{equation}%
		is completely monotonic on $\left( 0,\infty \right) $.
	\end{theorem}
	
	\begin{proof}
		Let 
		\begin{equation*}
			\mathcal{\tilde{R}}\left( x\right) =\mathcal{R}\left( x\right) -\frac{2}{x}%
			=-2\psi \left( x+1\right) -2\gamma .
		\end{equation*}%
		Then $\mathcal{\tilde{R}}\left( 0\right) =0$ and 
		\begin{equation*}
			\mathcal{\tilde{R}}^{\left( n\right) }\left( x\right) =-2\psi ^{\left(
				n\right) }\left( x+1\right) \text{ \ and }\mathcal{\tilde{R}}^{\left(
				n\right) }\left( 0\right) =-2\psi ^{\left( n\right) }\left( 1\right)
			=2\left( -1\right) ^{n}n!\zeta \left( n+1\right) .
		\end{equation*}%
		Hence, 
		\begin{equation*}
			\mathcal{\tilde{R}}\left( x\right) =\mathcal{R}\left( x\right) -\frac{2}{x}=%
			\mathcal{\tilde{R}}\left( 0\right) +\sum_{n=1}^{\infty }\frac{\mathcal{%
					\tilde{R}}^{\left( n\right) }\left( 0\right) }{n!}x^{n}=2\sum_{n=1}^{\infty
			}\left( -1\right) ^{n}\zeta \left( n+1\right) x^{n},
		\end{equation*}%
		which is clearly convergent on $\left( 0,1\right) $ (see also \cite[p. 259.
		Eq. (6.3.14)]{Abramowitz-HMFFGMT-1970}).
		
		Using the integral representation of remainder term of Taylor formula (\ref%
		{rtf-ir}) gives 
		\begin{align*}
			\mathcal{R}r_{n}\left( x\right) & =\frac{1}{n!}\int_{0}^{1}\left( 1-u\right)
			^{n}\left( -1\right) ^{n-1}\mathcal{\tilde{R}}^{\left( n+1\right) }\left(
			ux\right) du \\
			& =\frac{2}{n!}\int_{0}^{1}\left( 1-u\right) ^{n}\left( -1\right) ^{n}\psi
			^{\left( n+1\right) }\left( {ux+1}\right) du,
		\end{align*}%
		which is completely monotonic on $\left( 0,\infty \right) $. This completes
		the proof.
	\end{proof}
	
	Finally, we give the higher order monotonicity of the function%
	\begin{equation*}
		x\mapsto \mathcal{\tilde{R}}_{\theta ,n}\left( x\right) =\frac{1}{x^{\theta }%
		}\left( \mathcal{R}\left( x\right) -2\sum_{k=0}^{n-1}\frac{1}{k+x}\right) .
	\end{equation*}
	
	\begin{theorem}
		Let $n\in \mathbb{N}$ and $\theta \geq 1$. The function $x\mapsto \mathcal{%
			\tilde{R}}_{\theta ,n}^{\prime }\left( x\right) $ is completely monotonic on 
		$\left( 0,\infty \right) $. In particular, the function%
		\begin{equation*}
			x\mapsto \mathcal{\tilde{R}}_{1,1}\left( x\right) =\frac{x\mathcal{R}\left(
				x\right) -2}{x^{2}}
		\end{equation*}%
		is increasing from $\left( 0,\infty \right) $ onto $\left( -\pi
		^{2}/3,0\right) $.
	\end{theorem}
	
	\begin{proof}
		It is easy to verify that%
		\begin{align*}
			\mathcal{\tilde{R}}_{\theta ,n}\left( x\right) =& -2\frac{\psi \left(
				n+x\right) +\gamma }{x^{\theta }}, \\
			\left( -1\right) ^{k}\left( \frac{1}{x^{\theta }}\right) ^{\left( k\right)
			}=& \frac{\Gamma \left( \theta +k\right) }{\Gamma \left( \theta \right) }%
			\frac{1}{x^{\theta +k}},
		\end{align*}%
		and denote by $\psi _{k}\left( x\right) =\left( -1\right) ^{k-1}\psi
		^{\left( k\right) }\left( x\right) $. Let $m\in \mathbb{N}$. Differentiation
		yields%
		\begin{align*}
			\frac{\left( -1\right) ^{m-1}}{2}\mathcal{\tilde{R}}_{\theta ,n}^{\left(
				m\right) }\left( x\right) =& \left( -1\right) ^{m}\left( \frac{1}{x^{\theta }%
			}\right) ^{\left( m\right) }\left( \psi \left( n+x\right) +\gamma \right)  \\
			& +\left( -1\right) ^{m}\sum_{k=1}^{m}\binom{m}{k}\left( \frac{1}{x^{\theta }%
			}\right) ^{\left( m-k\right) }\left( \psi \left( n+x\right) +\gamma \right)
			^{\left( k\right) } \\
			=& \frac{\Gamma \left( \theta +m\right) }{\Gamma \left( \theta \right) }%
			\frac{\psi \left( n+x\right) +\gamma }{x^{\theta +m}} \\
			& -\sum_{k=1}^{m}\binom{m}{k}\frac{\Gamma \left( \theta +m-k\right) }{\Gamma
				\left( \theta \right) }\frac{1}{x^{\theta +m-k}}\psi _{k}\left( n+x\right) 
			\\
			:=& \frac{r_{\theta ,n}\left( x\right) }{x^{\theta +m}\Gamma \left( \theta
				\right) },
		\end{align*}%
		where%
		\begin{equation*}
			r_{\theta ,n}\left( x\right) =\Gamma \left( \theta +m\right) \left( \psi
			\left( n+x\right) +\gamma \right) -\sum_{k=1}^{m}\binom{m}{k}\Gamma \left(
			\theta +m-k\right) x^{k}\psi _{k}\left( n+x\right) 
		\end{equation*}%
		with $r_{\theta ,n}\left( 0\right) =\Gamma \left( \theta +m\right) \left(
		\psi \left( n\right) +\gamma \right) >0$ for all $m\in \mathbb{N}$.
		Differentiation again yields 
		\begin{eqnarray*}
			r_{\theta ,n}^{\prime }\left( x\right)  &=&\Gamma \left( \theta +m\right)
			\psi ^{\prime }\left( n+x\right) -\sum_{k=1}^{m}\binom{m}{k}\Gamma \left(
			\theta +m-k\right) kx^{k-1}\psi _{k}\left( n+x\right)  \\
			&&-\sum_{k=1}^{m}\binom{m}{k}\Gamma \left( \theta +m-k\right) x^{k}\psi
			_{k}^{\prime }\left( n+x\right) .
		\end{eqnarray*}%
		Since $\psi _{k}^{\prime }\left( x\right) =-\psi _{k+1}\left( x\right) $, by
		a shift of index, $r_{\theta ,n}^{\prime }\left( x\right) $ can be written
		as 
		\begin{eqnarray*}
			r_{\theta ,n}^{\prime }\left( x\right)  &=&\left[ \Gamma \left( \theta
			+m\right) -m\Gamma \left( \theta +m-1\right) \right] \psi _{1}\left(
			n+x\right) +\Gamma \left( \theta \right) x^{m}\psi _{m+1}\left( n+x\right) 
			\\
			&&+\sum_{k=2}^{m}\left[ \binom{m}{k-1}\Gamma \left( \theta +m-k+1\right) -k%
			\binom{m}{k}\Gamma \left( \theta +m-k\right) \right] x^{k-1}\psi _{k}\left(
			n+x\right) .
		\end{eqnarray*}%
		A simplication leads to%
		\begin{equation*}
			\Gamma \left( \theta +m\right) -m\Gamma \left( \theta +m-1\right) =\left(
			\theta -1\right) \Gamma \left( m+\theta -1\right) \geq 0,
		\end{equation*}%
		and%
		\begin{eqnarray*}
			&&\binom{m}{k-1}\Gamma \left( \theta +m-k+1\right) -k\binom{m}{k}\Gamma
			\left( \theta +m-k\right)  \\
			&=&\left( \theta -1\right) \binom{m}{k-1}\Gamma \left( \theta +m-k\right)
			\geq 0,
		\end{eqnarray*}%
		which implies that $r_{\theta ,n}^{\prime }\left( x\right) >0$ for $x>0$.
		Hence, $r_{\theta ,n}(x)>r_{\theta ,n}(0)=\Gamma (\theta +m)(\psi (n)+\gamma
		)>0$ for $x>0$, which results in $\left( -1\right) ^{m-1}\mathcal{\tilde{R}}%
		_{\theta ,n}^{\left( m\right) }\left( x\right) >0$ for all $m\in \mathbb{N}$
		and $x>0$, that is, $x\mapsto \mathcal{\tilde{R}}_{\theta ,n}^{\prime
		}\left( x\right) $ is completely monotonic on $\left( 0,\infty \right) $.
		Taking $\theta =n=1$ gives the increasing property of $\mathcal{\tilde{R}}%
		_{1,1}\left( x\right) $ with%
		\begin{equation*}
			\lim_{x\rightarrow 0^{+}}\mathcal{\tilde{R}}_{1,1}\left( x\right) =-\frac{\pi ^{2}}{%
				3}\text{ \ and \ }\lim_{x\rightarrow \infty }\mathcal{\tilde{R}}%
			_{1,1}\left( x\right) =0,
		\end{equation*}%
		which completes the proof.
	\end{proof}
	
	\section{Monotonicity property of $\mathcal{R}(x)/\mathcal{B}(x)$ and $%
		\left( \mathcal{B}(x)-\mathcal{R}\left( x\right) \right) /x^{2}$}\label{sec-4}
	
	In this section, we prove the monotonicity property of the functions%
	\begin{eqnarray}
		x &\mapsto &\frac{\mathcal{R}(x)}{\mathcal{B}(x)}=\frac{\Gamma \left(
			2x\right) }{\Gamma \left( x\right) ^{2}}\left( -2\psi \left( x\right)
		-2\gamma \right) ,  \label{R/B} \\
		x &\mapsto &D\left( x\right) =\frac{\mathcal{B}\left( x\right) -\mathcal{R}%
			\left( x\right) }{x^{2}}=\frac{1}{x^{2}}\left[ \frac{\Gamma \left( x\right)
			^{2}}{\Gamma \left( 2x\right) }-\left( -2\psi \left( x\right) -2\gamma
		\right) \right]  \label{D}
	\end{eqnarray}%
	on $\left( 0,\infty \right) $.
	
	\begin{theorem}
		\label{T-B/R-dcc}The function $x\mapsto \mathcal{R}(x)/\mathcal{B}(x)$
		strictly decreasing and concave on $\left( 0,\infty \right) $ with%
		\begin{equation*}
			\lim_{x\rightarrow 0+}\frac{\mathcal{R}(x)}{\mathcal{B}(x)}=1\text{ \ and \ }%
			\lim_{x\rightarrow \infty }\frac{\mathcal{R}(x)}{\mathcal{B}(x)}=-\infty .
		\end{equation*}%
		In particular, $x\mapsto \mathcal{R}(x)/\mathcal{B}(x)$\ is strictly
		decreasing and concave from $\left( 0,1\right) $\ onto $\left( 0,1\right) $.
	\end{theorem}
	
	\begin{theorem}
		\label{T-B-R/xx-pd}The function $D$ defined by (\ref{D}) is decreasing from $%
		\left( 0,\infty \right) $ onto $\left( 0,2\zeta \left( 3\right) \right) $.
	\end{theorem}
	
	In order to prove Theorems \ref{T-B/R-dcc} and \ref{T-B-R/xx-pd}, we need
	the following two lemmas.
	
	\begin{lemma}
		\label{L-h1-pd}Let 
		\begin{equation}
			h_{1}\left(x\right)=x\left(\psi\left(2x\right)-\psi\left( x\right)\right)+1.
			\label{h1}
		\end{equation}%
		Then $h_{1}\left( x\right) $ is positive and increasing on $\left( 0,\infty
		\right) $.
	\end{lemma}
	
	\begin{proof}
		Using the duplication formula (\ref{psi-df}) and integral representation (%
		\ref{psi-ir}) then integration by parts gives 
		\begin{align*}
			h_{1}(x)& =\frac{3}{2}+x\ln 2-\frac{1}{2}x\left[ \psi \left( x+1\right)
			-\psi \left( x+\frac{1}{2}\right) \right]  \\
			& =\frac{3}{2}+x\ln 2-\frac{1}{2}x\int_{0}^{\infty }\frac{e^{-t/2}-e^{-t}}{%
				1-e^{-t}}e^{-xt}dt \\
			& =\frac{3}{2}+x\ln 2+\frac{1}{2}\int_{0}^{\infty }\frac{1}{e^{t/2}+1}%
			de^{-xt} \\
			& =\frac{3}{2}+x\ln 2-\frac{1}{4}+\frac{1}{4}\int_{0}^{\infty }\frac{e^{t/2}%
			}{\left( 1+e^{t/2}\right) ^{2}}e^{-xt}dt.
		\end{align*}%
		Then%
		\begin{eqnarray*}
			h_{1}^{\prime }(x) &=&\ln 2-\frac{1}{4}\int_{0}^{\infty }\frac{te^{t/2}}{%
				\left( 1+e^{t/2}\right) ^{2}}e^{-xt}dt, \\
			h_{1}^{\prime \prime }(x) &=&\frac{1}{4}\int_{0}^{\infty }\frac{t^{2}e^{t/2}%
			}{\left( 1+e^{t/2}\right) ^{2}}e^{-xt}dt>0\text{ for }x>0\text{.}
		\end{eqnarray*}%
		It follows that 
		\begin{equation*}
			h_{1}^{\prime }(x)>\lim_{x\rightarrow 0^{+}}h_{1}^{\prime }(0)=\ln 2-\frac{1%
			}{2}\left[ \psi \left( 1\right) -\psi \left( \frac{1}{2}\right) \right] =0
		\end{equation*}%
		for $x>0$, which implies that 
		\begin{equation*}
			h_{1}\left( x\right) >\lim_{x\rightarrow 0^{+}}h_{1}\left( x\right) =\frac{3%
			}{2}>0
		\end{equation*}%
		for $x>0$. This completes the proof.
	\end{proof}
	
	\begin{lemma}
		\label{L-h2-pd}Let%
		\begin{equation}
			h_{2}\left( x\right) =2\left( \psi \left( x\right) +\gamma \right) \left(
			\psi \left( 2x\right) -\psi \left( x\right) \right) +\psi ^{\prime }\left(
			x\right) .  \label{h2}
		\end{equation}%
		Then $h_{2}^{\prime \prime }\left( x\right) <0$ for $x>0$. Consequently, $%
		h_{2}\left( x\right) ,h_{2}^{\prime }\left( x\right) >0$ for $x>0$.
	\end{lemma}
	
	\begin{proof}
		Differentiation yields 
		\begin{align*}
			h_{2}^{\prime }\left( x\right)=&2\left( \psi \left( x\right) +\gamma \right)
			\left( 2\psi ^{\prime }\left( 2x\right) -\psi ^{\prime }\left( x\right)
			\right) +2\psi ^{\prime }\left( x\right) \left( \psi \left( 2x\right) -\psi
			\left( x\right) \right) +\psi ^{\prime \prime }\left( x\right) , \\
			h_{2}^{\prime \prime }\left( x\right)=&2\left( \psi \left( x\right) +\gamma
			\right) \left( 4\psi ^{\prime \prime }\left( 2x\right) -\psi ^{\prime \prime
			}\left( x\right) \right) +4\psi ^{\prime }\left( x\right) \left( 2\psi
			^{\prime }\left( 2x\right) -\psi ^{\prime }\left( x\right) \right) \\
			&+2\psi ^{\prime \prime }\left( x\right) \left( \psi \left( 2x\right) -\psi
			\left( x\right) \right) +\psi ^{\prime \prime \prime }\left( x\right) .
		\end{align*}%
		Denote by 
		\begin{equation*}
			P_{0}=\psi \left( x+1\right) +\gamma ,\ P_{1}=\psi ^{\prime }\left(
			x+1\right) ,P_{2}=\psi ^{\prime \prime }\left( x+1\right) ,
		\end{equation*}%
		and%
		\begin{equation*}
			Q_{0}=\frac{1}{2}\psi \left( x+1\right) -\frac{1}{2}\psi \left( x+\frac{1}{2}%
			\right) ,\quad Q_{1}=\frac{1}{2}\psi ^{\prime }\left( x+1\right) -\frac{1}{2}%
			\psi ^{\prime }\left( x+\frac{1}{2}\right) ,
		\end{equation*}%
		\begin{equation*}
			Q_{2}=\frac{1}{2}\psi ^{\prime \prime }\left( x+1\right) -\frac{1}{2}\psi
			^{\prime \prime }\left( x+\frac{1}{2}\right) .
		\end{equation*}%
		Then combining (\ref{psi-df}) and (\ref{pn-df}), we can rewrite $%
		h_{2}^{\prime \prime }(x)$ as 
		\begin{align*}
			h_{2}^{\prime \prime }\left( x\right) & =2\left( P_{0}-\frac{1}{x}\right)
			\left( \frac{1}{x^{3}}-Q_{2}\right) +2\left( P_{2}-\frac{2}{x^{3}}\right)
			\left( \ln 2+\frac{1}{2x}-Q_{0}\right) \\
			& +4\left( P_{1}+\frac{1}{x^{2}}\right) \left( -Q_{1}-\frac{1}{2x^{2}}%
			\right) +\psi ^{\prime \prime \prime }\left( x\right) .
		\end{align*}
		Using the recurrence formula (\ref{pn-rf}), we obtain 
		\begin{align*}
			h_{2}^{\prime \prime }\left( x+1\right) & =2P_{0}\left( \frac{8}{\left(
				2x+1\right) ^{3}}-Q_{2}\right) +2P_{2}\left( \ln 2-Q_{0}+\frac{1}{2x+1}%
			\right) \\
			& +4P_{1}\left( -Q_{1}-\frac{2}{\left( 2x+1\right) ^{2}}\right) +\psi
			^{\prime \prime \prime }\left( x+1\right)
		\end{align*}%
		and then, 
		\begin{align*}
			h_{2}^{\prime \prime }\left( x+1\right) -h_{2}^{\prime \prime }\left(
			x\right) & =\frac{16}{\left( 2x+1\right) ^{3}}P_{0}-\frac{2}{x^{3}}P_{0}-%
			\frac{2}{x}Q_{2}+\frac{2}{x^{4}} \\
			& +\frac{4}{x^{3}}\ln 2+\frac{2}{2x+1}P_{2}-\frac{1}{x}P_{2}-\frac{4}{x^{3}}%
			Q_{0}+\frac{2}{x^{4}} \\
			& +\frac{2}{x^{2}}P_{1}-\frac{8}{4x^{2}+4x+1}P_{1}+\frac{4}{x^{2}}Q_{1}+%
			\frac{2}{x^{4}}-\frac{6}{x^{4}},
		\end{align*}%
		which, by those notations $P_{i}$ and $Q_{i}$ for $i=0,1,2$, can be arranged
		as 
		\begin{align*}
			h_{2}^{\prime \prime }\left( x+1\right) -h_{2}^{\prime \prime }\left(
			x\right) & =\dfrac{2}{x^{3}}\left( \psi \left( x+\dfrac{1}{2}\right) +\gamma
			+2\ln 2\right) -\dfrac{2}{x^{2}}\psi ^{\prime }\left( x+\dfrac{1}{2}\right)
			\bigskip \\
			& +\dfrac{1}{x}\psi ^{\prime \prime }\left( x+\dfrac{1}{2}\right) -\dfrac{%
				4(4x^{3}+12x^{2}+6x+1)}{x^{3}\left( 2x+1\right) ^{3}}\left( \psi \left(
			x+1\right) +\gamma \right) \bigskip \\
			& +\dfrac{4(2x^{2}+4x+1)}{x^{2}\left( 2x+1\right) ^{2}}\psi ^{\prime }\left(
			x+1\right) -\dfrac{2(x+1)}{x\left( 2x+1\right) }\psi ^{\prime \prime }\left(
			x+1\right) .
		\end{align*}
		
		Now let 
		\begin{align}
			h_{3}\left( x\right) & =x^{3}\left[ h_{2}^{\prime \prime }\left( x+1\right)
			-h_{2}^{\prime \prime }\left( x\right) \right]   \label{h3} \\
			& =2\left[ \psi \left( x+\frac{1}{2}\right) +\gamma +2\ln 2\right] -2x\psi
			^{\prime }\left( x+\frac{1}{2}\right) +x^{2}\psi ^{\prime \prime }\left( x+%
			\frac{1}{2}\right)   \notag \\
			& -\frac{4\left( 4x^{3}+12x^{2}+6x+1\right) }{\left( 2x+1\right) ^{3}}\left(
			\psi \left( x+1\right) +\gamma \right)   \notag \\
			& +\frac{4x\left( 2x^{2}+4x+1\right) }{\left( 2x+1\right) ^{2}}\psi ^{\prime
			}\left( x+1\right) -\frac{2x^{2}\left( x+1\right) }{2x+1}\psi ^{\prime
				\prime }\left( x+1\right) .  \notag
		\end{align}%
		Then%
		\begin{equation*}
			h_{3}\left( 0^{+}\right) =\lim_{x\rightarrow 0^{+}}h_{3}\left( x\right)
			=0,\quad h_{3}\left( \infty \right) =\lim_{x\rightarrow \infty }h_{3}\left(
			x\right) =4\ln 2,
		\end{equation*}%
		and by differentiation%
		\begin{align}
			h_{3}^{\prime }\left( x\right) =& x^{2}\psi ^{\prime \prime \prime }\left( x+%
			\frac{1}{2}\right) +\frac{48x^{2}}{\left( 2x+1\right) ^{4}}\left( \psi
			\left( x+1\right) +\gamma \right)   \label{h4} \\
			& -\frac{24x^{2}}{\left( 2x+1\right) ^{3}}\psi ^{\prime }\left( x+1\right) +%
			\frac{6x^{2}}{\left( 2x+1\right) ^{2}}\psi ^{\prime \prime }\left(
			x+1\right)   \notag \\
			& -\frac{2x^{2}\left( x+1\right) }{\left( 2x+1\right) }\psi ^{\prime \prime
				\prime }\left( x+1\right)   \notag \\
			:=& \frac{x^{2}}{\left( 2x+1\right) ^{4}}h_{4}\left( x\right) ,  \notag
		\end{align}%
		where 
		\begin{align}
			h_{4}\left( x\right) =& \left( 2x+1\right) ^{4}\psi ^{\prime \prime \prime
			}\left( x+\frac{1}{2}\right) +48\left( \psi \left( x+1\right) +\gamma
			\right)   \notag \\
			& -24\left( 2x+1\right) \psi ^{\prime }\left( x+1\right) +6\left(
			2x+1\right) ^{2}\psi ^{\prime \prime }\left( x+1\right)   \notag \\
			& -2\left( x+1\right) \left( 2x+1\right) ^{3}\psi ^{\prime \prime \prime
			}\left( x+1\right) ,  \notag \\
			h_{4}\left( 0^{+}\right) =& \frac{13}{15}\pi ^{4}-4\pi ^{2}-12\zeta \left(
			3\right) =30.518\cdots ,\quad h_{4}\left( \infty \right) =\infty ,  \notag
		\end{align}%
		\begin{align}
			\dfrac{1}{\left( 2x+1\right) ^{4}}h_{4}^{\prime }\left( x\right) =& \dfrac{8%
			}{2x+1}\psi ^{\prime \prime \prime }\left( x+\dfrac{1}{2}\right) +\psi
			^{\left( 4\right) }\left( x+\dfrac{1}{2}\right)   \label{h5} \\
			& -\dfrac{8}{2x+1}\psi ^{\prime \prime \prime }\left( x+1\right) -\dfrac{%
				2(x+1)}{\left( 2x+1\right) }\psi ^{\left( 4\right) }\left( x+1\right)  
			\notag \\
			:=& h_{5}\left( x\right) ,  \notag
		\end{align}%
		\begin{equation*}
			h_{5}\left( 0^{+}\right) =\frac{112}{15}\pi ^{4}-696\zeta \left( 5\right)
			=5.619\cdots >0,\quad h_{5}\left( \infty \right) =0,
		\end{equation*}%
		\begin{align}
			& \frac{1}{2}\left( 2x+3\right) \left( 2x+1\right) \left[ h_{5}\left(
			x+1\right) -h_{5}\left( x\right) \right]   \label{h6} \\
			& =8\psi ^{\prime \prime \prime }\left( x+1\right) +\psi ^{\left( 4\right)
			}\left( x+1\right) -8\psi ^{\prime \prime \prime }\left( x+\frac{3}{2}%
			\right) -\frac{24(2x+1)}{\left( x+1\right) ^{5}}  \notag \\
			& :=h_{6}\left( x\right) ,  \notag
		\end{align}%
		and 
		\begin{equation}
			h_{6}\left( x+1\right) -h_{6}\left( x\right) =\frac{%
				24(80x^{4}+560x^{3}+1480x^{2}+1750x+781)}{\left( 2x+3\right) ^{4}\left(
				x+2\right) ^{5}}>0  \label{h6d}
		\end{equation}%
		for $x>0$.
		
		Equation (\ref{h6d}) implies that 
		\begin{equation*}
			h_{6}\left( x\right) <h_{6}\left( x+1\right) <\cdot \cdot \cdot
			<\lim_{n\rightarrow \infty }h_{6}\left( x+n\right) =0
		\end{equation*}%
		for $x>0$. By (\ref{h6}), it yields that $h_{5}\left( x+1\right)
		-h_{5}\left( x\right) <0$ for $x>0$, and so 
		\begin{equation*}
			h_{5}\left( x\right) >h_{5}\left( x+1\right) >\cdot \cdot \cdot
			>\lim_{n\rightarrow \infty }h_{5}\left( x+n\right) =0
		\end{equation*}%
		for $x>0$. This in combination with (\ref{h5}) leads to that $h_{4}^{\prime
		}\left( x\right) >0$ for $x>0$, and therefore, 
		\begin{equation*}
			h_{4}\left( x\right) >h_{4}\left( 0^{+}\right) =\frac{13}{15}\pi ^{4}-4\pi
			^{2}-12\zeta \left( 3\right) >0\text{ \ for }x>0.
		\end{equation*}%
		From (\ref{h4}) it is seen that $h_{3}^{\prime }\left( x\right) >0$ for $x>0$%
		, which together with $h_{3}\left( 0^{+}\right) =0$ indicates that $%
		h_{3}\left( x\right) >0$ for $x>0$. By the relation (\ref{h3}), it is found
		that $h_{2}^{\prime \prime }\left( x+1\right) -h_{2}^{\prime \prime }\left(
		x\right) >0$ for $x>0$, and it follows that 
		\begin{equation*}
			h_{2}^{\prime \prime }\left( x\right) <h_{2}^{\prime \prime }\left(
			x+1\right) <\cdot \cdot \cdot <\lim_{n\rightarrow \infty }h_{2}^{\prime
				\prime }\left( x+n\right) =0
		\end{equation*}%
		for $x>0$. This yields%
		\begin{equation*}
			h_{2}^{\prime }\left( x\right) >\lim_{x\rightarrow \infty }h_{2}^{\prime
			}\left( x\right) =0\text{ \ for }x>0\text{,}
		\end{equation*}%
		which, together with 
		\begin{align*}
		  {h_{2}\left( 0^+\right)} & =\lim_{x\rightarrow 0^{+}}
			\left[ 2\left( \psi (x+1)+\gamma -\frac{1}{x}\right) \left( \psi (2x+1)-\psi
			(x+1)+\frac{1}{2x}\right) +\psi ^{\prime }(x+1)+\frac{1}{x^{2}}\right]  \\
			& =\lim_{x\rightarrow 0^{+}}\left[ \psi ^{\prime }(x+1)+\frac{\psi
				(x+1)+\gamma }{x}+\frac{2\left( \psi (2x+1)-\psi (x+1)\right) }{x}\right]  \\
			& =\psi ^{\prime }(1)+\psi ^{\prime }(1)-2\psi ^{\prime }(1)=0,
		\end{align*}%
		in turn leads to $h_{2}\left( x\right) >0$ for $x>0$, thereby completing the
		proof.
	\end{proof}
	
	We are now in a position to prove Theorem \ref{T-B/R-dcc}.
	
	\begin{proof}[Proof of Theorem \protect\ref{T-B/R-dcc}]
		Differentiation yields%
		\begin{align*}
			\left[ \frac{\mathcal{R}(x)}{\mathcal{B}(x)}\right] ^{\prime }=& -2\frac{%
				\Gamma \left( 2x\right) }{\Gamma \left( x\right) ^{2}}h_{2}\left( x\right) ,
			\\
			\left[ \frac{\mathcal{R}(x)}{\mathcal{B}(x)}\right] ^{\prime \prime }=& -2%
			\left[ \frac{\Gamma \left( 2x\right) }{\Gamma \left( x\right) ^{2}}\right]
			^{\prime }h_{2}\left( x\right) -2\frac{\Gamma \left( 2x\right) }{\Gamma
				\left( x\right) ^{2}}h_{2}^{\prime }\left( x\right) , \\
			\left[ \frac{\Gamma \left( 2x\right) }{\Gamma \left( x\right) ^{2}}\right]
			^{\prime }=& 2\frac{\Gamma \left( 2x\right) }{\Gamma \left( x\right) ^{2}}%
			\left( \psi \left( 2x\right) -\psi \left( x\right) \right) >0\text{ for }x>0,
		\end{align*}%
		where $h_{2}\left( x\right) $ is defined by (\ref{h2}). By Lemma \ref%
		{L-h2-pd}, $h_{2}\left( x\right) ,h_{2}^{\prime }\left( x\right) >0$ for $x>0
		$. Then $\left( \mathcal{R}/\mathcal{B}\right) ^{\prime },\left( \mathcal{R}/%
		\mathcal{B}\right) ^{\prime \prime }<0$ for $x>0$. Clearly, $%
		\lim_{x\rightarrow 1}{\mathcal{R}(x)}/{\mathcal{B}(x)}=0$, and an easy
		computation gives 
		\begin{equation*}
			\lim_{x\rightarrow 0^{+}}\frac{\mathcal{R}(x)}{\mathcal{B}(x)}%
			=\lim_{x\rightarrow 0^{+}}\frac{x\left( -\psi (x)-\gamma \right) }{\Gamma
				(x+1)^{2}/\Gamma (2x+1)}=\lim_{x\rightarrow 0^{+}}x\left( -\psi (x)-\gamma
			\right) =1
		\end{equation*}%
		and 
		\begin{align*}
			\lim_{x\rightarrow \infty }\frac{\mathcal{R}(x)}{\mathcal{B}(x)}&
			=\lim_{x\rightarrow \infty }\frac{\Gamma (2x)}{\Gamma (x)^{2}}\left( -2\psi
			(x)-2\gamma \right) =\lim_{x\rightarrow \infty }\frac{\sqrt{2\pi }%
				e^{-2x}(2x)^{2x-1/2}\left( -2\psi (x)-2\gamma \right) }{2\pi e^{-2x}x^{2x-1}}
			\\
			& =\lim_{x\rightarrow \infty }\frac{2^{2x-1/2}\sqrt{x}\left( -2\psi
				(x)-2\gamma \right) }{\sqrt{2\pi }}=-\infty ,
		\end{align*}%
		thereby completing the proof.
	\end{proof}
	
	Lastly, we prove Theorem \ref{T-B-R/xx-pd}.
	
	\begin{proof}[Proof of Theorem \protect\ref{T-B-R/xx-pd}]
		Let%
		\begin{equation*}
			f\left( x\right) =1-\left( -2\psi \left( x\right) -2\gamma \right) \frac{%
				\Gamma \left( 2x\right) }{\Gamma \left( x\right) ^{2}}\text{ \ and \ }%
			g\left( x\right) =x^{2}\frac{\Gamma \left( 2x\right) }{\Gamma \left(
				x\right) ^{2}}.
		\end{equation*}%
		Then $D\left( x\right) =f\left( x\right) /g\left( x\right) $ with $f\left(
		0^{+}\right) =g\left( 0^{+}\right) =0$. Differentiation yields 
		\begin{align*}
			f^{\prime }\left( x\right) =& 2\frac{\Gamma \left( 2x\right) }{\Gamma \left(
				x\right) ^{2}}\left[ 2\left( \psi \left( x\right) +\gamma \right) \left(
			\psi \left( 2x\right) -\psi \left( x\right) \right) +\psi ^{\prime }\left(
			x\right) \right] =2\frac{\Gamma \left( 2x\right) }{\Gamma \left( x\right)
				^{2}}h_{2}\left( x\right) , \\
			g^{\prime }\left( x\right) =& 2\frac{\Gamma \left( 2x\right) }{\Gamma \left(
				x\right) ^{2}}x\left[ x\left( \psi \left( 2x\right) -\psi \left( x\right)
			\right) +1\right] =2\frac{\Gamma \left( 2x\right) }{\Gamma \left( x\right)
				^{2}}xh_{1}(x),
		\end{align*}%
		where $h_{1}\left( x\right) $ and $h_{2}\left( x\right) $ are defined by (%
		\ref{h1}) and (\ref{h2}), respectively. Then 
		\begin{equation*}
			\frac{f^{\prime }\left( x\right) }{g^{\prime }\left( x\right) }=\frac{1}{%
				h_{1}\left( x\right) }\times \frac{h_{2}\left( x\right) }{x}.
		\end{equation*}%
		It has been proved in Lemma \ref{L-h1-pd} that $h_{1}\left( x\right) $ is
		positive and increasing on $\left( 0,\infty \right) $. Hence, $1/h_{1}\left(
		x\right) $ is positive and decreasing on $\left( 0,\infty \right) $. Since $%
		h_{2}\left( 0^{+}\right) =0$ and $h_{2}^{\prime \prime }\left( x\right) <0$
		for $x>0$ by Lemma \ref{L-h2-pd}, it follows from Lemma \ref{L-LMR} together
		with the limiting value $\lim_{x\rightarrow \infty }h_{2}(x)/x=0$ that the
		function $h_{2}\left( x\right) /x$ is also positive and decreasing on $%
		\left( 0,\infty \right) $. Then $f^{\prime }\left( x\right) /g^{\prime
		}\left( x\right) $ is decreasing on $\left( 0,\infty \right) $. Applying
		Lemma \ref{L-LMR} again, the decreasing property of $D\left( x\right) $ on $%
		\left( 0,\infty \right) $ follows immediately. This completes the proof.
	\end{proof}
	
	\section{An affirmative answer to Conjecture \protect\ref{C-Qiu}}\label{sec-5}
	
	In this section, we give an answer to Conjecture \ref{C-Qiu} made by Qiu, Ma
	and Huang in \cite{Qiu-CA-51-2020}.
	
	\begin{theorem}
		Let $R\left( x\right) =R\left( x,1-x\right) $ and $B\left( x\right) =B\left(
		x,1-x\right) $. The function 
		\begin{equation*}
			x\mapsto F\left( x\right) =R\left( x\right) -\frac{B\left( x\right) }{%
				1+x\left( 1-x\right) }
		\end{equation*}%
		is completely monotonic on $\left( 0,1/2\right) $.
	\end{theorem}
	
	\begin{proof}
		It was proved in \cite[Lemma 2.2]{Qiu-CA-51-2020} that 
		\begin{align*}
			R\left( x\right)&=\ln 16+4\sum_{n=1}^{\infty }2^{2n}\lambda \left(
			2n+1\right) \left( \frac{1}{2}-x\right) ^{2n}=\sum_{n=0}^{\infty
			}u_{n}t^{2n}, \\
			B\left( x\right)&=4\sum_{n=0}^{\infty }2^{2n}\beta \left( 2n+1\right) \left( 
			\frac{1}{2}-x\right) ^{2n}=\sum_{n=0}^{\infty }v_{n}t^{2n},
		\end{align*}%
		with $t=1/2-x$, $u_{0}=\ln 16$ and 
		\begin{align*}
			u_{n}&=2^{2n+2}\lambda \left( 2n+1\right) \text{ \ for }n\geq 1\text{,} \\
			v_{n}&=2^{2n+2}\beta \left( 2n+1\right) \text{ \ for }n\geq 0.
		\end{align*}%
		Here $\lambda \left( 2n+1\right) $ and $\beta \left( 2n+1\right) $ are given
		by (\ref{ln}) and (\ref{b2n+1}), respectively. Using the formulas given in (%
		\ref{ln}) and (\ref{b2n+1}), $u_{n}$ and $v_{n}$ can be written as 
		\begin{align*}
			u_{n}&=2\left( 2^{2n+1}-1\right) \zeta \left( 2n+1\right) \text{ \ for }%
			n\geq 1\text{,} \\
			v_{n}&=\frac{\pi ^{2n+1}}{\left( 2n\right) !}\left\vert E_{2n}\right\vert 
			\text{ \ for }n\geq 0\text{.}
		\end{align*}
		
		Noting that $1+x\left( 1-x\right) =5/4-t^{2}$ and using the Cauchy product
		formula, we get 
		\begin{equation*}
			\frac{B\left( x\right) }{1+x\left( 1-x\right) }=\frac{4}{5}\frac{%
				\sum_{n=0}^{\infty }v_{n}t^{2n}}{1-4t^{2}/5}=\sum_{n=0}^{\infty }\left( 
			\frac{4}{5}\sum_{k=0}^{n}\left( \frac{4}{5}\right) ^{n-k}v_{k}\right) t^{2n},
		\end{equation*}%
		and thereby 
		\begin{equation*}
			F\left( x\right) =R\left( x\right) -\frac{B\left( x\right) }{1+x\left(
				1-x\right) }=\sum_{n=0}^{\infty }s_{n}t^{2n},
		\end{equation*}%
		where 
		\begin{equation*}
			s_{n}=u_{n}-\left( \frac{4}{5}\right) ^{n+1}\sum_{k=0}^{n}\left( \frac{5}{4}%
			\right) ^{k}v_{k}.
		\end{equation*}%
		In what follows, it is enough to show that $s_{n}\geq 0$ for all $n\in 
		\mathbb{N}_{0}$, since $t^{2n}=\left( 1/2-x\right) ^{2n}$ is completely
		monotonic in $x$ on $\left( 0,1/2\right) $. In fact, on one hand, a simple
		verification gives 
		\begin{equation*}
			s_{0}=u_{0}-\frac{4}{5}v_{0}=\ln{16}-\frac{4}{5}\pi E_{0}=\ln 16-\frac{4}{5}%
			\pi =0.259...>0,
		\end{equation*}%
		and on the other hand, applying the second inequality of (\ref{E2n<}) to $%
		v_{k}$, we obtain 
		\begin{equation*}
			v_{k}=\frac{\pi ^{2k+1}}{\left( 2k\right) !}\left\vert E_{2k}\right\vert
			<4^{k+1}\text{\ for \ }k\in \mathbb{N}_{0},
		\end{equation*}%
		and hence, for $n\geq 1$, 
		\begin{align*}
			s_{n}& =u_{n}-\left( \frac{4}{5}\right) ^{n+1}\sum_{k=0}^{n}\left( \frac{5}{4%
			}\right) ^{k}v_{k}>u_{n}-\left( \frac{4}{5}\right)
			^{n+1}\sum_{k=0}^{n}\left( \frac{5}{4}\right) ^{k}4^{k+1} \\
			& =2^{2n+2}\left( 1-2^{-2n-1}\right) \zeta \left( 2n+1\right)
			-2^{2n+2}\left( 1-5^{-n-1}\right) \\
			& >2^{2n+2}-2^{2n+2}\left( 1-5^{-n-1}\right) =2^{2n+2}5^{-n-1}>0.
		\end{align*}%
		The last second inequality follows from the known functional inequality $%
		\left( 1-2^{-x}\right) \zeta \left( x\right) >1$ for $x>1$. This completes
		the proof.
	\end{proof}
	
	\section{Concluding remarks}
	
	In this paper, we presented several expansions and higher order monotonicity
	properties involving the functions%
	\begin{equation*}
		\mathcal{B}\left( x\right) =\frac{\Gamma \left( x\right) ^{2}}{\Gamma \left(
			2x\right) }\text{ \ and \ }\mathcal{R}\left( x\right) =-2\psi \left(
		x\right) -2\gamma ;
	\end{equation*}%
	we also found that the function $x\mapsto \mathcal{R}\left( x\right) /%
	\mathcal{B}\left( x\right) $ is decreasing and concave on $\left( 0,\infty
	\right) $, while $x\mapsto \left( \mathcal{B}\left( x\right) -\mathcal{R}%
	\left( x\right) \right) /x^{2}$ is decreasing on $\left( 0,\infty \right) $;
	Moreover, we gave an affirmative answer to Conjecture \ref{C-Qiu}.
	
	Finally, we present several remarks.
	
	\begin{remark}
		In Section \ref{sec-3}, we do not present asymptotic expansions of $\mathcal{B}\left(
		x\right) $. In fact, by (\ref{Bl}) and \cite[Theorem 1]{Chen-AMC-250-2015},
		we have%
		\begin{equation*}
			\ln \left( \frac{x}{\sqrt{\pi }}2^{2x-1}\mathcal{B}\left( x\right) \right)
			=\ln \frac{\Gamma \left( x+1\right) }{\Gamma \left( x+1/2\right) }\thicksim 
			\frac{1}{2}\ln x+\sum_{k=1}^{n}\frac{\left( 1-2^{-2k}\right) B_{2k}}{k\left(
				2k-1\right) }\frac{1}{x^{2k-1}}
		\end{equation*}%
		as $x\rightarrow \infty $, and the function%
		\begin{equation*}
			x\mapsto \left( -1\right) ^{n}\left[ \ln \mathcal{B}\left( x\right) -\frac{1%
			}{2}\ln \pi +\left( 2x-1\right) \ln 2+\frac{1}{2}\ln x-\sum_{k=1}^{n}\frac{%
				\left( 1-2^{-2k}\right) B_{2k}}{k\left( 2k-1\right) }\frac{1}{x^{2k-1}}%
			\right] 
		\end{equation*}%
		is completely monotonic on $\left( 0,\infty \right) $. Using \cite[Corollary
		5]{Yang-PAMS-148-2020}, we have%
		\begin{equation*}
			\ln \left( \frac{x}{\sqrt{\pi }}2^{2x-1}\mathcal{B}\left( x\right) \right)
			\thicksim \frac{1}{2}\ln \left( x+\frac{1}{4}\right) +\sum_{k=1}^{n}\frac{%
				B_{2k+1}\left( 1/4\right) }{k\left( 2k+1\right) \left( x+1/4\right) ^{2k}}
		\end{equation*}%
		as $x\rightarrow \infty $, and the function%
		\begin{eqnarray*}
			x &\mapsto &\left( -1\right) ^{n}\left[
			\begin{array}{l}
			 \ln \mathcal{B}\left( x\right) -%
			\dfrac{1}{2}\ln \pi +\left( 2x-1\right) \ln 2+\ln x \\
			\quad -\dfrac{1}{2}\ln \left( x+\dfrac{1}{4}\right) -\dsum\limits_{k=1}^{n}\dfrac{
				B_{2k+1}\left( 1/4\right) }{k\left( 2k+1\right) \left( x+1/4\right) ^{2k}}
			\end{array}
			\right] 
		\end{eqnarray*}%
		is completely monotonic on $\left( 0,\infty \right) $.
	\end{remark}
	
	\begin{remark}
		The reciprocal of $\mathcal{B}\left( x\right) $ has a hypergeometric series
		representation in \cite[Remark 6]{Tian-JMAA-493-2021} that 
		\begin{equation*}
			\frac{1}{\mathcal{B}\left( x\right) }=\dfrac{\Gamma \left( 2x\right) }{%
				\Gamma \left( x\right) ^{2}}=\sum_{k=1}^{\infty }\frac{\left( -x\right)
				_{k}^{2}}{k!\left( k-1\right) !}.
		\end{equation*}
	\end{remark}
	
	\begin{remark}
		From those higher order monotonicity results given in this paper, we can
		deduce some new inequalities for $\mathcal{B}\left( x\right) $ and $\mathcal{%
			R}\left( x\right) $. For instance, 
		\begin{enumerate}[leftmargin=2.2em,label=(\roman*)]
			\item by Theorem \ref{T-RlnxB-sc} (iii), the double inequality 
		\begin{align*}
			&\frac{2}{x}\exp \left[ \sum_{k=2}^{2n}\left( -1\right) ^{k-1}\left(
			2^{k}-2\right) k^{-1}\zeta \left( k\right) x^{k}\right] \\
			&\hspace{1cm}<\mathcal{B}\left( x\right) <\frac{2}{x}\exp \left[ \sum_{k=2}^{2n+1}%
			\left( -1\right) ^{k-1}\left( 2^{k}-2\right) k^{-1}\zeta \left( k\right)
			x^{k}\right]
		\end{align*}%
		holds for $x>0$ and $n\in \mathbb{N}$;
		\item by the decreasing property of $\mathcal{B}r_{1}\left( x\right) $
		defined by (\ref{Br1}), we have%
		\begin{equation*}
			2^{1-2x}\sum_{k=0}^{n}\frac{W_{k}}{x+k}<\mathcal{B}\left( x\right)
			<2^{1-2x}\sum_{k=0}^{n}\frac{W_{k}}{x+k}+\lambda _{n}2^{1-2x}
		\end{equation*}%
		for $x>0$ and $n\in \mathbb{N}$, where $\lambda _{n}=\sum_{k=n+1}^{\infty
		}\left( W_{k}/k\right) $;
		\item  by the complete monotonicity of $\mathcal{R}r_{n}\left( x\right) $
		defined by (\ref{Rrn}), the double inequality%
		\begin{equation*}
			\frac{2}{x}+2\sum_{k=1}^{2n-1}\left( -1\right) ^{k}\zeta \left( k+1\right)
			x^{k}<\mathcal{R}\left( x\right) <\frac{2}{x}+2\sum_{k=1}^{2n}\left(
			-1\right) ^{k}\zeta \left( k+1\right) x^{k}
		\end{equation*}%
		for $x>0$ and $n\in \mathbb{N}$;
		\item using the decreasing and concave property of $x\mapsto \mathcal{R}(x)/%
		\mathcal{B}(x)$ on $\left( 0,1\right) $ (given in Theorem \ref{T-B/R-dcc}),
		we have%
		\begin{equation*}
			1-x<\frac{\mathcal{R}(x)}{\mathcal{B}(x)}<1\text{ \ for }x\in \left(
			0,1\right) \text{;}
		\end{equation*}
		\item  by an application of the decreasing property of $x\mapsto \left( \mathcal{B}%
		\left( x\right) -\mathcal{R}\left( x\right) \right) /x^{2}$ shown in
		Theorems \ref{T-B-R/xx-pd}, we have%
		\begin{equation*}
			4\left( \pi -4\ln 2\right) x^{2}<\mathcal{B}\left( x\right) -\mathcal{R}%
			\left( x\right) <2\zeta \left( 3\right) x^{2}
		\end{equation*}%
		for $x\in \left( 0,1/2\right) $.
			\end{enumerate}
	\end{remark}
	
	\begin{remark}
		Some power and hypergeometric series representations given in Section 3
		imply certain identities. For example, putting $x\rightarrow 1/2$ in (\ref%
		{lnB-ps}) gives%
		\begin{equation*}
			\sum_{n=2}^{\infty }\left( -1\right) ^{n}\frac{\left( 1-2^{1-n}\right) \zeta
				\left( n\right) }{n}=\ln \frac{4}{\pi };
		\end{equation*}%
		letting $x\rightarrow 1/2$ in (\ref{R-ps}) gives%
		\begin{equation*}
			\sum_{n=1}^{\infty }\left( -1\right) ^{n-1}2^{-n}\zeta \left( n+1\right)
			=2-\ln 4.
		\end{equation*}
	\end{remark}
	
	\begin{remark}
		An interesting infinite series appeared in Theorem \ref{T-B-hs}, that is,%
		\begin{equation}
			\mathcal{H}\left( s\right) :=\sum_{n=1}^{\infty }\frac{W_{n}}{n^{s}}.
			\label{H}
		\end{equation}%
		By the inequalities (\ref{Wn<>}), the series $\mathcal{H}\left( s\right) $
		is convergent for $s>1/2$. From the proof of Theorem \ref{T-B-hs} or Remark %
		\ref{Remark2}, we see that%
		\begin{equation*}
			\frac{\sqrt{\pi }\Gamma \left( x\right) }{\Gamma \left( x+1/2\right) }-\frac{%
				1}{x}=\sum_{n=1}^{\infty }\frac{W_{n}}{x+n},
		\end{equation*}%
		which, by letting $x\rightarrow 0^+$, gives%
		\begin{equation*}
			\sum_{n=1}^{\infty }\frac{W_{n}}{n}=\ln 4.
		\end{equation*}%
		Similarly,%
		\begin{equation*}
			\sum_{n=1}^{\infty }\frac{W_{n}}{n^{2}}=-\lim_{x\rightarrow 0^+}\frac{d}{dx}%
			\left( \frac{\sqrt{\pi }\Gamma \left( x\right) }{\Gamma \left( x+1/2\right) }%
			-\frac{1}{x}\right) =\frac{\pi ^{2}}{6}-2(\ln 2)^2.
		\end{equation*}%
		In general, for $k\in \mathbb{N}_{0}$, we have%
		\begin{equation*}
			\sum_{n=1}^{\infty }\frac{W_{n}}{n^{k+1}}=\frac{\left( -1\right) ^{k}}{k!}%
			\lim_{x\rightarrow 0^+}\frac{d^{k}}{dx^{k}}\left( \frac{\sqrt{\pi }\Gamma
				\left( x\right) }{\Gamma \left( x+1/2\right) }-\frac{1}{x}\right) .
		\end{equation*}
	\end{remark}
	
	Further, we propose the following problem.
	
	\begin{problem}
		Discuss the properties of $\mathcal{H}\left( s\right) :=\sum_{n=1}^{\infty
		}\left( W_{n}/n^{s}\right) $ and compute $\mathcal{H}\left( s\right) $ for
		certain special $s>1/2$.
	\end{problem}

\end{document}